\documentclass[reqno]{amsart}

\usepackage{amsmath,amssymb,amsthm}

\usepackage{tikz}

\usepackage{caption}
\usepackage{graphicx}
\usepackage{graphics}
\usepackage{algorithm}
\usepackage{algorithmicx}
\usepackage{algpseudocode}

\newtheorem{theorem}{Theorem}

\newtheorem{lemma}{Lemma}

\begin{document}

\title[]{ Balanced stick breaking}

\author[]{Fran\c{c}ois Cl\'ement \and Stefan Steinerberger}

\address{Department of Mathematics, University of Washington, Seattle}
 \email{fclement@uw.edu }
 \email{steinerb@uw.edu}

\begin{abstract} 
 Consider an infinite sequence $(x_k)_{k=1}^{\infty}$ on the unit circle $\mathbb{S}^1$. We may interpret the first $n$ elements $(x_k)_{k=1}^{n}$ as places where the `circular stick' $\mathbb{S}^1$ is broken into a total of $n+1$ pieces. It is clear that they cannot all be the same length all the time. de Bruijn and Erd\H{o}s (1949) show that the ratio of the largest to the smallest has to be arbitrarily close to 2 infinitely many times which is sharp. They also consider the problem of balancing the length of $r$ consecutive intervals and prove 
 $$ \frac{\max \mbox{length of}~r~\mbox{consecutive intervals}}{\min \mbox{length of}~r~\mbox{consecutive intervals}} \geq 1 + \frac{1}{r}.$$
 We prove that this ratio can be as small as $1 + c \log{r}/ r$. This is done by means of refined discrepancy estimates for the van der Corput sequence over very short intervals and proves a conjecture of Brethouwer.
\end{abstract}

\subjclass[2020]{}

\maketitle

\vspace{-0pt}

\section{Introduction and Results}
\subsection{Introduction}
Let $(x_k)_{k=1}^{\infty}$ be an infinite sequence on the unit circle $\mathbb{S}^1 \cong [0,1)$. We interpret the sequence as places where the `circular stick' $\mathbb{S}^1$ is being broken. After the first $n$ steps $(x_k)_{k=1}^{n}$ of the procedure, we have $n+1$ intervals.
Naturally, these sticks cannot all have the same length at every step of the process and the question is whether this can be made precise. This was made quantitative by de Bruijn and Erd\H{o}s in three different ways.
\begin{center}
    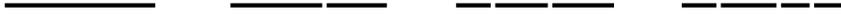
\begin{figure}[h!]
      \begin{tikzpicture}[scale=2]
          \draw [ultra thick] (0,0) -- (1,0);
          \draw [ultra thick] (2-0.5,0) -- (2.61-0.5,0);
          \draw [ultra thick] (2.64-0.5,0) -- (3.04-0.5,0);
          \draw [ultra thick]  (3, 0) -- (3.23, 0);
           \draw [ultra thick]  (3.26, 0) -- (3.61, 0);
            \draw [ultra thick]  (3.64, 0) -- (4.05, 0);
     \draw [ultra thick]  (4.5, 0) -- (4.73, 0);     
     \draw [ultra thick]  (4.76, 0) -- (5.13, 0);  
     \draw [ultra thick]  (5.16, 0) -- (5.35, 0);  
     \draw [ultra thick]  (5.38, 0) -- (5.6, 0);       
      \end{tikzpicture}
      \caption{Iteratively breaking a stick.}
    \end{figure}
\end{center}

\begin{theorem}[de Bruijn and Erd\H{o}s \cite{deb}, 1949]
For any sequence $(x_k)_{k=1}^{\infty}$ we have
\begin{align*}
    \limsup_{n \rightarrow \infty} \quad n \cdot (\emph{longest interval after}~n~\emph{steps}) &\geq \frac{1}{\log{2}} \\
 \limsup_{n \rightarrow \infty} \quad n \cdot (\emph{shortest interval after}~n~\emph{steps}) &\leq \frac{1}{\log{4}}
 \end{align*}
 and
 $$
  \limsup_{n \rightarrow \infty} \quad \frac{ \emph{longest interval after}~n~\emph{steps}}{ \emph{shortest interval after}~n~\emph{steps}} \geq 2.$$
\end{theorem}
Theorem 1 was also established by Ostrowski \cite{ost0, ost}, Sch\"onhage \cite{schonhage} and Toulmin \cite{toulmin}. de Bruijn and Erd\H{o}s show that all three results are sharp and that an extremal example for all three is given by the sequence
$$ x_k = \log_2(2k-1) \mod 1.$$
Basic properties of the logarithm make an analysis of this sequence fairly straightforward \cite{ram}.
Variations were considered by Anholcer, Bosek, Grytczuk, Gutowski,  Przybylo, Pyzik and Zajac \cite{anh}, Chung-Graham \cite{chung0, chung, chung1, chung2} and Groemer \cite{groemer}.

\subsection{Main result} One very natural variation already considered in the original paper is to not consider a single interval but instead $r$ consecutive intervals.  de Bruijn and Erd\H{o}s prove that
 $$\limsup_{n \rightarrow \infty} \quad \frac{ \mbox{largest length of}~r~\mbox{consecutive intervals}}{   \mbox{smallest length of}~r~\mbox{consecutive intervals}} \geq 1 + \frac{1}{r}.$$
The case $r=1$ recovers their original result. This problem appears to be completely open for every $r \geq 2$. 
de Bruijn and Erd\H{o}s  conjecture that their lower bound is \textit{not} optimal and that $1/r$ can be replaced by $f(r)/r$ where $f(r) \rightarrow \infty$ as $r \rightarrow \infty$. We show that $f(r)$ cannot grow faster than $\log{r}$.

\begin{theorem}
    There exists a sequence $(x_k)_{k=1}^{\infty}$ in $[0,1] \cong \mathbb{S}^1$ and a universal constant $0 < c < \infty$ such that for all $r \in \mathbb{N}$ and all $n \in \mathbb{N}$ sufficiently large (depending on $r$)  the first $n$ elements of the sequence satisfy
    $$ \frac{ \emph{largest length of}~r~\emph{consecutive intervals}}{   \emph{smallest length of}~r~\emph{consecutive intervals}} \leq 1 + \frac{c\log{r}}{r}.$$
\end{theorem}

We give two structurally very different examples which are both classical in the study of irregularity of distributions.  Our first example is the sequence of irrational rotations $\left\{ \phi \right\},  \left\{ 2\phi \right\},  \left\{ 3\phi \right\}, \left\{ 4\phi \right\}, \dots$
where $\left\{x\right\} = x - \left\lfloor x \right\rfloor$ is the fractional part of $x$ and
$ \phi = (1+\sqrt{5})/2$
is the golden ratio. This sequence is very well-understood. The $r=1$ case of Theorem 2 is implied by a detailed study of the gaps undertaken by Ravenstein \cite{raven}, see also the 2014 PhD thesis of Habib \cite{habib}. Our second example is the van der Corput sequence
$$\frac12, \frac14, \frac34, \frac18, \frac58, \frac38, \frac78, \dots$$
Brethouwer in his 2024 PhD thesis \cite{bret} did an extensive numerical analysis of the van der Corput sequence and explicitly conjectured Theorem 2. He also conjectured that the rate in Theorem 2 might be optimal.

\subsection{Discrepancy at small scales} The Kronecker sequence $(\left\{k \phi\right\})_{k=1}^{\infty}$ and the van der Corput sequence are two of the best-known sequences in $[0,1]$ and have been intensively studied. In particular, if $(x_k)_{k=1}^{\infty}$ denotes either of the two sequences, then it is known that, for all $0 \leq x \leq 1$ and all $n \in \mathbb{N}$
$$ \# \left\{1 \leq k \leq n:  0 \leq x_k \leq x \right\} = x n + \mathcal{O}(\log{n}).$$
A celebrated result of Schmidt \cite{schmidt} shows that the error term $\mathcal{O}(\log{n})$ is the smallest possible error term for any sequence. Theorem 2 follows from showing that we can obtain a tighter error bound if we restrict ourselves to very short intervals.

\begin{theorem} Let $(x_k)_{k=1}^{\infty}$ denote either the van der Corput sequence in base 2 or the Kronecker sequence $x_k = \left\{ k \phi\right\}$. There exists a universal constant $c>0$ so that for all $r \in \mathbb{N}$ and all $n \in \mathbb{N}$ sufficiently large (depending on $r$) and all $0 \leq x \leq 1-r/n$
    $$ \left| \# \left\{1 \leq k \leq n:  x \leq x_k \leq x + \frac{r}{n} \right\} - r \right| \leq c \cdot \log{r}.$$
\end{theorem}
We actually show a slightly stronger result: one can consider all intervals of length $r/n$ on $\mathbb{S}^1$. This implies Theorem 2.  We prove Theorem 3 for the van der Corput sequence in \S 2 and for the Golden Ratio Kronecker sequence $x_k = \left\{k \phi\right\}$ in \S 3. These two proofs are completely independent of each other and use very different structures and arguments. It is very conceivable that these arguments could be extended to the van der Corput sequence in base $b$ and the Kronecker sequence $x_k = \left\{k \alpha \right\}$ for any badly approximable $\alpha \in \mathbb{R}$.\\ 
It would be interesting to have structurally different types of examples exhibiting this type of regularity at very short scales. We are not aware of any results in that direction. There is the related concept of pair correlation: consider 
$$ F_N(s) = \frac{1}{N} \# \left\{ 1 \leq m \neq n \leq N: \|x_m - x_n\| \leq \frac{s}{N} \right\}.$$
For example, a sequence is said to have Poissonian Pair Correlation if, for all $s>0$, we have that
$ \lim_{N \rightarrow \infty} F_N(s) = 2s.$
This purely local property is known to have global implications \cite{aist, stein0, stein}. The typical example of a sequence with Poissonian Pair Correlation is a sequence of independent and uniformly distributed random variables which then has the desired property almost surely. Sequences like the van der Corput sequence and the Kronecker sequences have a rather rigid gap structure and it is easy to see that they do not exhibit Poissonian Pair Correlation. However, in light of Theorem 3, they satisfy the structural property that for any $s>0$ and all $N \in \mathbb{N}$ sufficiently large
$ \left| F_N(s) - 2s \right| \leq c \log{s}.$

\section{Proof for the van der Corput sequence}
\subsection{An Ordering Lemma}
We begin by showing a fundamental fact: the van der Corput sequence in base 2 is `more dense' on the left-side of the interval than anywhere else in the interval. We start with a motivating example: set $x_0 = 0$ and consider the first $660$ elements of the van der Corput sequence ordered by size 
$$ 0 < \frac{1}{1024} < \frac{1}{512} < \frac{1}{256} < \frac{5}{1024} < \frac{3}{512} < \frac{1}{128} < \dots < \frac{511}{512}.$$
\begin{center}
\begin{figure}[h!]
    \includegraphics[width=0.6\textwidth]{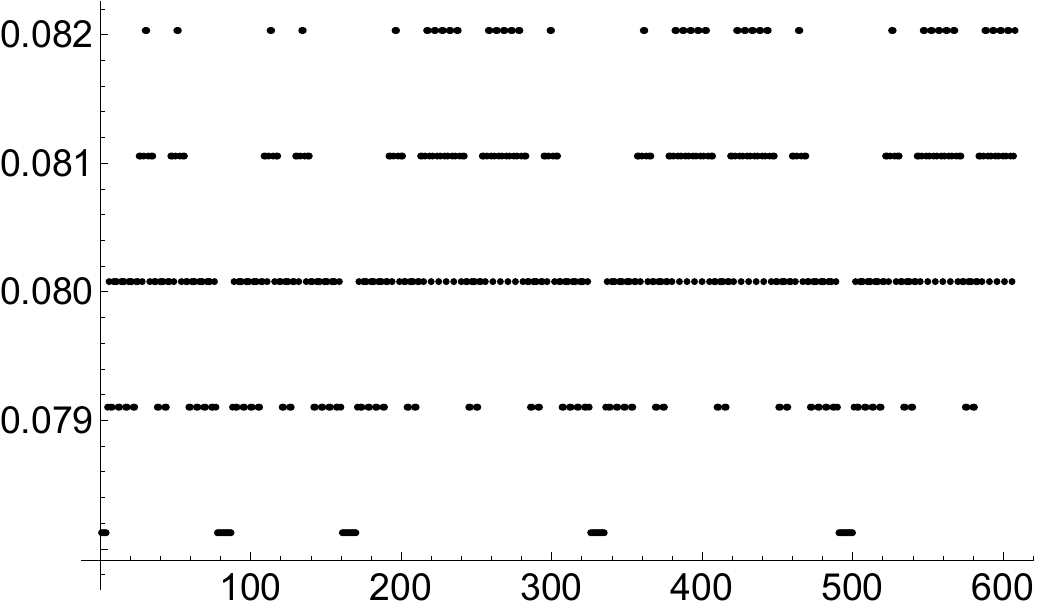}
\caption{Setting $x_0 = 0$ and taking the first 660 elements of the van der Corput sequence, the plot shows $x_{i+53} - x_i$. }
    \end{figure}
\end{center}
We consider, chosen more or less randomly, the sequence $x_{i+53} - x_i$, the result is shown in Fig. 2. As it turns out, we indeed have
$$ x_{i+53} - x_i \geq x_{53} = \frac{5}{64} = 0.078125$$
but the actual behavior of $x_{i+53} - x_i$ is somewhat intricate and equality is attained several times, not only for $i=0$.  However, equality is also attained for $i=0$ which justifies the notion that `the interval at the very left contains the most points'.
We can now make this precise. Throughout this section we interpret indices in the usual cyclic manner: $x_{j}=x_{j \mod n}$ for $j>n$. Our arguments automatically include intervals $[x_i,x_{i+j}]$ such that $[x_i,x_{i+j}]=[x_i,0] \cup [0,x_{i+j}]$, they do not need to be treated as a separate case. The main ingredient of the proof is the following statement that says that `the densest region' containing the most elements of the van der Corput sequence is always the left interval while `the emptiest region' is always given by the right end of the interval. It will be convenient to set $x_0 = 0$.
\begin{lemma}[Main Lemma]\label{prop:left} Suppose $0= x_0 < x_1 < \dots < x_{n}<1$ are the first $n$ elements of the van der Corput sequence in $[0,1]$ sorted in increasing order. Then for any $i \in \{0,\ldots,i\}$ and any $r \in \mathbb{N}$,
$$x_r-0\leq x_{i+r}-x_i \leq 1 - x_{n+1-r}.$$
\end{lemma}

Before giving the proof in \S 2.2, we quickly discuss the main idea behind. We start by noting that the result is clearly true when $n=2^k-1$ since, in that case,
$$ \left\{x_0,x_1, x_2, \dots, x_{2^k-1} \right\} = \left\{ 0,\frac{1}{2^k}, \frac{2}{2^k}, \dots, \frac{2^k-1}{2^k} \right\}$$
and any interval $[x_i,x_{i+1}]$ has length $1/2^k$. When $2^k \leq n \leq 2^{k+1} - 1$, then each interval $[x_i,x_{i+1}]$ can only have two different possible lengths: it is either ``short'', meaning a length equal to $1/2^{k+1}$, or it is ``long'' and has a length of $1/2^{k}$. One way of thinking about it is as follows: for $2^k \leq n \leq 2^{k+1} - 1$ the new elements are being added exactly in the middle of intervals created by the first $2^{k} - 1$ points. Each interval is split exactly in the middle and each interval is either `split' or `unsplit'.
For the proof, we will need several minor results and some terminology. We use \emph{intervals} to refer to an interval of the form $[x_i,x_{i+j}]$. The smallest segments $[x_i,x_{i+1}]$, the gaps between two consecutive (ordered) points that form the intervals, will be referred to as \emph{elements}. We use the term `canonical interval of length $2^{-h}$' for any interval of the shape $[a/2^{h},(a+1)/2^{h}]$ for some integer $a$ and $0\leq h \leq k+1$. For a fixed $n$ and $k=\lfloor\log_2(n)\rfloor$, we use $x_i$ for the $i$-th element of the sorted van der Corput set with $n$ points, and $x'_i$ for those of the sorted van der Corput set with $2^{k}-1$ points.\\

The main idea behind the proof is the following: we first show in Lemma~\ref{lem:shift} that the minimal length of $r$ consecutive elements can only be obtained for an interval $[x_i,x_{i+r}]$ for which $x_i=a/2^k$, with $k=\lfloor \log_2(n)\rfloor$. Given such an interval $[x_i,x_{i+r}]$, there exists a smallest interval $[x'_g,x'_{g+j}]$ such that $[x_i,x_{i+r}]\subseteq[x'_g,x'_{g+j}]$. Note that $x'_g=x_i$ and $x'_{g+j}=x_{i+r}$ or $x'_{g+j}=x_{i+r}+2^{-k-1}$. Lemma~\ref{lem:induc} then matches all the elements of $[x'_g,x'_{g+j}]$ with distinct elements of $[0,x'_j]$ such that whenever an element of $[x'_g,x'_{g+j}]$ is split by an incoming point, then the corresponding element of $[0,x_j]$ has also been split previously (or they are the same). We are then ready for the proof of Lemma~\ref{prop:left}: this guarantees that any short element in $[x_i,x_{i+r}]$ can be matched with one in some interval $[0,x_t]$, with $x_t\geq x_r$. $[0,x_r]$ has as many elements as $[x_i,x_{i+r}]$ and at least as many of these are small. This proves the first inequality, the second inequality is then shown to be dual to the first.

\subsection{Proof of the Main Lemma}
\begin{lemma}\label{lem:shift}
    For any $x_t$ such that $x_t=a/2^{k+1}$ with $a$ odd, the interval $[x_t,x_{t+r}]$ is at least as long as $[x_{t-1},x_{t+r-1}]$.
\end{lemma}

\begin{proof}
    The two intervals differ by one element at each endpoint. $[x_{t-1},x_t]$ is a ``short'' interval as $x_t$ is one of the points added between $2^k$ and $2^{k+1}$ points: $[x_{t-1},x_{t+1}]$ was an element of length $2^{-k}$ that was then split in two halves by the addition of $x_t$. Regardless of the interval on the right, $[x_{t-1},x_{t-1+r}]$ cannot be longer than $[x_t,x_{t+r}]$.
\end{proof}

Suppose we are given a uniform partition induced by a number of the form $n= 2^{k}-1$. We want to show that, as we keep adding points, whenever a point is being added somewhere, we can identify this with another point that is being added earlier on the left side. This is, unsurprisingly, a consequence of the binary expansion of integers. We start with the following purely technical result.

\begin{lemma}\label{lem:flip}
Suppose we have placed $2^k-1$ points. Let $[x'_l,x'_{l+1}]$ and $[x'_m,x'_{m+1}]$ be two elements such that $x'_l$ is written $A0C$ and $x'_m$ $B1C$ in inverse binary decomposition, where $A$ and $B$ are some binary prefixes of identical length, and $C$ is some binary suffix. Then $[x'_l,x'_{l+1}]$ is split before $[x'_m,x'_{m+1}]$ when going to $2^{k+1}-1$ points.
\end{lemma}

\begin{proof}
    $[x'_l,x'_{l+1}]$ will be split when we obtain the integer whose inverse binary decomposition is $A0C1$ (of total length $k+1$), therefore in binary that integer is $1r(C)0r(A)$, where $r(A)$ is the reversed binary sequence $A$. Similarly, $[x'_m,x'_{m+1}]$ will be split by the integer whose binary decomposition is $1r(C)1r(B)$. Regardless of $A$ and $B$, the first integer is smaller than the second, and therefore $[x'_l,x'_{l+1}]$ is split first.
\end{proof}

We can now formulate the most important technical statement that induces a bijection between the quantities of interest.

\begin{lemma}\label{lem:induc}
Suppose $0= x'_0\leq x'_1 \ldots\leq x'_{2^t-1}<1$ are the first $2^t-1$ points of the van der Corput sequence in $[0,1]$, sorted in increasing order.
    Let $[x'_g,x'_{g+j}]$ be some interval of length $0\leq j\leq r$ and $[0,x'_j]$ the corresponding initial interval of equal length. 
    Then there exists a bijection from the elements of $[x'_g,x'_{g+j}]$ to those of $[0,x_j']$, such that for any $[x'_s,x'_{s+1}]$ in $[x_g',x'_{g+j}]$ the associated element of $[0,x'_j]$ is cut before $[x'_s,x'_{s+1}]$ as we go to $2^{t-1}-1$ points.
\end{lemma}

\begin{proof}[Proof of Lemma~\ref{lem:induc}]
We show the property by induction on $t$.
If $t=0$, we only have one possible interval: $[0,1]$, the proposition is trivially true.
Suppose the Lemma is true for $t-1$. Then this means that given the first $2^{t-1}$ points of the van der Corput sequence in increasing order, for any interval $[x'_{g},x'_{g+j}]$, we can pair each element with an element of $[0,x'_j]$ that will be cut before it as we add points until having $2^t-1$.
Let $0=x'_0< x'_1< \ldots< x'_{2^t-1}<1 $ and let $[x'_g,x'_{g+j}]$ be an interval of length $j$. Notice that the $2^t$ points that get added afterwards until we get $2^{t+1}-1$ points are a copy of the first $2^t$ (including 0), shifted to the right by $1/2^{t+1}$. Furthermore, the first $2^{t-1}$ are a copy of the first $2^{t-1}$ shifted by $1/2^{t+1}$, and the next $2^{t-1}$ a copy of the first $2^{t-1}$ shifted by $3/2^{t+1}$. 
For any given $h$, canonical intervals can be split in two categories: \emph{even} and \emph{odd} intervals. The even intervals correspond to the $[a/2^h,(a+1)/2^h]$ for which $a$ is even, and the odd intervals to the case where $a$ is odd. All the elements of $[0,1]$ for $2^t-1$ points correspond exactly to canonical intervals of length $2^{-t}$. We can separate the elements $[x'_s,x'_{s+1}]$ of $[x'_g,x'_{g+j}]$ and $[0,x'_j]$ between odd and even elements based on the $t$-th digit (the last one) of the binary decomposition of their left endpoint $x'_s$. Note that $[0,x'_j]$ always has at least as many even elements as there are odd elements as it starts with an even element. 

If we consider all the even elements of length $1/2^{t}$ of $[0,1]$, we can multiply their lengths by 2 to obtain exactly the exactly the element decomposition of $[0,1]$ that we had for $2^{t-1}$ points. We can do the same process with the odd elements, except they also need to be shifted by an extra $1/2^{t}$ to the left. In both cases, the ordering of the elements stays the same, but we can now apply the induction hypothesis on the even and odd intervals respectively.
There are now two cases to consider:
\begin{enumerate}
    \item If $[x'_g,x'_{g+j}]$ and $[0,x'_j]$ have the same number of even and odd elements, then we can directly use the induction hypothesis to match the even intervals of $[x'_g,x'_{g+j}]$ with those of $[0,x'_j]$. The odd elements behave exactly the same way except they are offset by $1/2^{t}$: we can apply again the induction hypothesis and match the odd elements of $[x'_g,x'_{g+j}]$ with those of $[0,x'_j]$. We have the desired assignment.
    \item If they have different numbers of odd and even, then $[0,x'_j]$ must have one extra even element and one less odd element compared to $[x'_g,x'_{g+j}]$. Using a similar reasoning with two induction hypotheses as in the previous case, there is a matching between the even intervals of $[0,x'_{j-2}]$ and all those of $[x'_g,x'_{g+j}]$. There is also a matching between the odd intervals of $[x'_g,x'_{g+j}]$ and those of $[0,x'_j]\cup [x'_j,x'_{j+1}]$ (this extra interval has to be odd since the last interval of $[0,x'_j]$ is even). For the odd interval of $[x'_g,x'_{g+j}]$ that is matched with $[x'_j,x'_{j+1}]$, assign it instead to the last even interval $[x'_{j-1},x'_{j}]$. This is a valid assignment by Lemma~\ref{lem:flip} as the left endpoints of $[x'_{j-1},x'_{j}]$ and $[x'_j,x'_{j+1}]$ can be written $A0$ and $A1$ in inverse binary decomposition respectively using $t$ bits after the `0.'. Indeed, they are consecutive canonical intervals of length $2^{-t}$, with the first being even. We have the desired assignment, this concludes the proof.
\end{enumerate}
\end{proof}

The statement is true for $t=k=\lfloor \log_2(n) \rfloor$. We now prove the Main Lemma.

\begin{proof}[Proof of the Main Lemma]
Let $0=x_0<x_1<\ldots<x_{n}<1$ be the sorted first $n$ points of the van der Corput sequence and $k=\lfloor \log_2(n)\rfloor$. Let $[x_i,x_{i+r}]$ be an interval of $r$ elements. By Lemma~\ref{lem:shift}, we can suppose $x_i=a/2^k$ with $a \in \mathbb{N}$ when searching for the interval $[x_i,x_{i+r}]$ with minimal length. Given $0=x'_0<x'_1<\ldots<x_{2^k-1}'<1$, the sorted first $2^k-1$ points of the van der Corput sequence, Lemma~\ref{lem:induc} states that for $[x_g',x'_{g+j}]$, the smallest interval containing $[x_i,x_{i+r}]$, each element $[x'_s,x'_{s+1}]$ can be matched with an element of $[0,x'_j]$ which will be split before it. Remember that $x'_g=x_i$, and $x'_{g+j}=x_{i+r}$ or $x'_{g+j}=x_{i+r}+2^{-k-1}$.

When going from $2^k$ points to $n<2^{k+1}$, we know by Lemma~\ref{lem:induc} that the number of short elements of length $2^{-k-1}$ in $[0,x'_j]$ is always \emph{at least} the number of short elements in $[x'_g,x'_{g+j}]$. When reaching $n$ points, there are two options:
\begin{enumerate}
    \item $[x_i,x_{i+r}]=[x'_g,x'_{g+j}]$. $[0,x'_j]$ has at least as many short elements as $[x_i,x_{i+r}]$, therefore $x_j'\leq x_{i+r}-x_i$. We also have that any long element in $[0,x'_j]$ corresponds to a long interval in $[x_i,x_{i+r}]$ by Lemma~\ref{lem:induc}, therefore there are at least $r$ elements in $[0,x'_j]$: $x_r\leq x'_j$.
    \item If $[x_i,x_{i+r}]=[x'_g,x'_{g+j}-2^{-k-1}]$. We have removed a single small element from $[x'_g,x'_{g+j}]$. The same arguments as in the previous case hold: $[0,x'_{j}]$ has at least as many short elements as $[x'_{g},x_{g+j}']$, and at most as many long elements. $[x_i,x_{i+r}]$ therefore contains at least as many long elements and at most one less short interval than $[0,x'_j]$. $[0,x_j']$ contains at least as many elements as $[x'_g,x'_{g+j}]$, which contains one more element than $[x_i,x_{i+r}]$. Since $[x_i,x_{i+r}]$ contains $r$ elements, $[0,x_j']$ has to contain at least $r+1$, and we have $x'_j-2^{-k-1}\leq x_{i+r}-x_i$. $[0,x_r] \subsetneq [0,x'_j]$: removing the last element of $[0,x'_j]$ removes at least a short interval of length $2^{-k-1}$: $x_r\leq x_j'-2^{-k-1}$.
\end{enumerate}
In both cases, we obtain $x_r\leq x_{i+r}-x_i$, which shows that the left block of $r$ elements, $[0,x_r]$, is the shortest contiguous interval of $r$ elements. We now show that the right block of $r$ elements, $[x_{n+1-r},1]$, is the longest. In other words, if 
$$ 0 = x_0 < x_1 < \dots < x_{n-1} < 1$$
are the first $n$ elements of the van der Corput sequence, then
$$  x_{i+r} - x_i \leq 1 - x_{n+1-r}.$$

The inequality is obtained by reversing the process of adding points. Rather than adding integers $n$ from $2^k$ to $2^{k+1}$, we add them in decreasing order from $2^{k+1}$ to $2^k$. This is exactly the same process as before, except that the roles of $0$ and $1$ in the binary decompositions have been switched. 
One can then apply exactly the same reasoning as before to show that the interval with $r$ elements starting in $1\ldots1$ (in reverse binary decomposition) and going leftwards will always be shorter than any other interval of length $r$. This means that for the regular process of adding points in, for any $[x_i,x_{i+r}]$ and $[x_{n+1-r},1]$, we can match the elements of $[x_i,x_{i+r}]$ to those of $[x_{n+1-r},1]$ such that the element of $[x_i,x_{i+r}]$ is always cut first. This implies that $1-x_{n+1-r}\geq x_{i+r}-x_i$ and concludes the proof of the Main Lemma.
\end{proof}

\subsection{Proof of the Theorem for van der Corput}
We are now ready to prove Theorem 3 for the van der Corput sequence. We quickly comment at the end of \S 2 how Theorem 3 implies Theorem 2.
\begin{proof}
Fix a value of $r \in \mathbb{N}$ and let $n$ be sufficiently large.
Suppose we are dealing with the first $n$ elements of the van der Corput sequence and $2^k-1 \leq n < 2^{k+1}-1$. When $n = 2^k-1$, then the points are perfectly equispaced at distance $2^{-k}$. In particular, the individual gaps between consecutive elements are all the same length and the statement is trivially true.
  It remains to deal with the case 
$2^k-1 < n < 2^{k+1}-1$.
The Main Lemma ensures that it is sufficient to control the number of elements in the densest interval (on the left) and the sparsest interval (on the right). We start on the left and will control the number of points in $[0,x]$. Since $r$ is fixed and $n$ is very large, it suffices to deal with $x = \ell/2^k$ for some fixed $\ell \in \mathbb{N}$ (which we may assume to satisfy $r/10 \leq \ell \leq 10r$) and $n$ sufficiently large. When $n = 2^k-1$, the interval $[0, \ell/2^k]$ contains exactly $\ell$ nonzero elements of the van der Corput sequence. Now we add the remaining $n - (2^k -1)$ points. We define $a \in \mathbb{N}$ as the unique integer such that
$$ 2^a \leq \ell < 2^{a+1}.$$
The elements of the van der Corput sequence in base 2 that end up in the interval $[0, \ell/2^k)$ also end up in the interval $[0, 2^{a+1}/2^k)$ because 
$[0, \ell/2^k) \subset [0, 2^{a+1}/2^k)$. The elements of the van der Corput sequence $(x_m)_{m=1}^{\infty}$ are contained in $[0, 2^{a+1}/2^k]$ whenever $m$ is a multiple of $2^{k-a -1}$. This means that the elements $x_m$ of the van der Corput sequence contained in $[0, 2^{a+1}/2^k]$ for $2^k \leq m \leq n$ are exactly
$$ x_{2^k}, \quad x_{2^k + 2^{k-a-1}}, \quad x_{2^k + 2 \cdot 2^{k-a-1}}, \quad x_{2^k + 3\cdot 2^{k-a-1}}, \cdots$$
This means that we have to count
$$ X = \# \left\{ m \in \mathbb{N}_{\geq 0}: 2^k + m \cdot 2^{k-a-1} \leq n \quad \mbox{and} \quad x_{2^k + m \cdot 2^{k-a-1}} \leq \frac{\ell}{2^k} \right\}.$$
The first inequality is easy to resolve and the problem can be rewritten as
$$X =  \# \left\{ 0 \leq m \leq \frac{n-2^k}{2^{k-a-1}}: \quad x_{2^k + m \cdot 2^{k-a-1}} \leq \frac{\ell}{2^k} \right\}.$$
At this point, we invoke the self-similar structure of the van der Corput sequence: an expansion in binary shows that for all $1 \leq m < 2^{a+1}$ we have
$$ \frac{2^k}{2^{a+1}} \cdot x_{2^k + m \cdot 2^{k-a-1}} = x_m + \frac{1}{2^{a+2}}.$$
Therefore, using this self-similarity, the set in question can be written as
\begin{align*}
    X &= \# \left\{ 0 \leq m \leq \frac{n-2^k}{2^{k-a-1}}: \quad \frac{2^k}{2^{a+1}} \cdot x_{2^k + m \cdot 2^{k-a-1}} \leq \frac{\ell}{2^{a+1}} \right\} \\
    &= \# \left\{ 0 \leq m \leq \frac{n-2^k}{2^{k-a-1}}: \quad x_m \leq \frac{\ell - 1/2}{2^{a+1}}  \right\}.
\end{align*}
At this point, we can invoke the standard-discrepancy estimate for the van der Corput sequence ensuring that, for all $A \in \mathbb{N}$ and $ 0 < B < 1$ we have
$$ \# \left\{1 \leq m \leq A: x_m \leq B \right\} = A \cdot B + \mathcal{O}(\log B).$$
This shows that
\begin{align*}
    X &=  \frac{n-2^k}{2^{k-a-1}} \cdot \frac{\ell - 1/2}{2^{a+1}} + \mathcal{O}\left(\log\left(  \frac{n-2^k}{2^{k-a-1}}\right) \right).
\end{align*}
We recall that
$$  \frac{n-2^k}{2^{k-a-1}} \leq \frac{2^{k}}{2^{k-a-1}} \leq 2^{a+1} \leq 10 \ell. $$
This shows that the total number of elements in $[0, \ell/2^k]$ is
\begin{align*}
    \# \left\{1 \leq m \leq n: x_m \leq \frac{\ell}{2^k} \right\} &= \# \left\{1 \leq m \leq 2^k-1: x_m \leq \frac{\ell}{2^k} \right\} \\
    &+\left\{2^k -1 \leq m \leq n: x_m \leq \frac{\ell}{2^k} \right\} \\
    &= \ell +  \frac{n-2^k}{2^{k}} (\ell - 1/2) + \mathcal{O}(\log(10\ell))
\end{align*}
which simplifies to $n \cdot \ell + \mathcal{O}(\log(10\ell))$ as desired. This concludes the argument for intervals of the length $[0, \ell/2^k]$. The argument on the other side, intervals of length $[1 - \ell/2^k, 1]$ is completely analogous: in fact, if an element $x_m$ is in that interval, then $x_{m+1}$ is in $[0, \ell/2^k]$ and the previous argument applies.
\end{proof}

Theorem 3 implies Theorem 2 in the following manner: let us assume the elements are sorted and the shortest interval is length $x_{k+r} - x_k$. Then the interval $[x_k, x_{k+r}]$ contains $r+1$ points. Theorem 2 implies that the length of the interval cannot be shorter than $r - c \log{r}$.  Likewise, if $x_{k+r} - x_k$ was a very long interval, then Theorem 2 implies that it cannot be longer than $r + c\log{r}$. This implies Theorem 2.

\section{Proof for the Golden Ratio Kronecker sequence}

\subsection{A Two Gap Lemma} We use $F_1 = F_2 = 1$ and $F_n = F_{n-1} + F_{n-2}$ for $n\geq 3$ to denote the Fibonacci numbers.  It is well understood that for any Kronecker sequence, there are at most 3 different gaps between two consecutive elements (also known as the Three Gap Theorem \cite{sos, sur, swi}). It is also well understood that this can be interpreted from the point of view of the continued fraction expansion \cite{ravv}. We first show that our desired statement is true for a suitable subsequence.

\begin{lemma}[Folklore] \label{lem:twogap}
    The set of $F_k$ points
    $$X = \left\{0, \left\{ \phi \right\},\left\{ 2\phi \right\}, \left\{ 3\phi \right\}, \dots, \left\{ (F_k-1) \phi \right\} \right\} \subset [0,1]$$
    satisfies, for all intervals $[x, x+\varepsilon] \subset \mathbb{S}^1$,
    $$ \# \left( X \cap \left[x, x + \varepsilon \right]\right) = F_k \cdot \varepsilon + \mathcal{O}(1).$$
\end{lemma}
\begin{proof} We use 
$$ \left| \frac{F_{k+1}}{F_k} - \phi \right| \leq \frac{1}{F_k^2}$$
 to write
\begin{align*}
     \left\{m \phi \right\} =\left\{m \left( \frac{F_{k+1}}{F_k} + \mathcal{O}(F_k^{-2}) \right) \right\}  = \left\{  \frac{F_{k+1} m}{F_k}  \right\}  + \mathcal{O}(F_k^{-1}).
\end{align*}
$F_k$ and $F_{k+1}$ are coprime and thus
$$ \left\{   \frac{F_{k+1} m}{F_k}: 0 \leq m \leq F_k - 1 \right\} = \left\{0, \frac{1}{F_k}, \dots, \frac{F_k-1}{F_k} \right\}.$$
Since the error is uniformly $\mathcal{O}(F_k^{-1})$, the result follows.
\end{proof}

\subsection{Proof of the Theorem for Golden Ratio Kronecker}
\begin{proof}
    
Consider the first $n$ elements of the sequence. We pick an element $x_i = \left\{a \phi \right\}$ and ask ourselves about how many of the first $n$ points are contained in $[x_i, x_i + r/n]$ where $r \in \mathbb{N}$ is a fixed integer and $n$ is assumed to be sufficiently large (depending on $r$). We work on $\mathbb{S}^1$ and understand everything cyclically.

\begin{center}
\begin{figure}[h!]
    \begin{tikzpicture}
        \draw [thick] (-1,0) -- (5,0);
          \draw [dashed, thick] (5,0) -- (6,0);
                    \draw [dashed, thick] (-2,0) -- (-1,0);    
      \filldraw (0,0) circle (0.06cm);  
       \filldraw (1,0) circle (0.06cm);  
       \filldraw (2,0) circle (0.06cm);  
          \filldraw (3,0) circle (0.06cm);  
      \filldraw (4,0) circle (0.06cm);   
      \node at (0, -0.3) {$\left\{a \phi\right\}$};
       \node at (4, -0.3) {$\left\{b \phi\right\}$};
       \draw [thick, <->] (0, -0.6) -- (4, -0.6);
       \node at (2, -0.9) {$r/n$};
    \end{tikzpicture}
\end{figure}  
\end{center}
\vspace{-20pt}

The goal is to understand, for an interval $J \subset \mathbb{S}^1$ of length $r/n$, the number of points in $J$. We may assume w.l.o.g. that the left end-point of $J$ is given by one of the points $x_i$ since that changes the total number of points in the interval at most by 1. The problem is then to analyze
$$ X = \left\{1 \leq m \leq n: x_i \leq \left\{ m \phi \right\} \leq x_i + \frac{r}{n} \right\}.$$
Using that $x_i = \left\{a \phi \right\}$, we can shift the problem and write
$$ X = \left\{1-a \leq m \leq n-a: 0 \leq \left\{ m \phi \right\} \leq  \frac{r}{n} \right\}.$$
This shows that the desired estimate for $\# X$ would follow from having, for $0 \leq A,B \leq n$ the uniform estimates
$$ Y= \# \left\{ 0 \leq m \leq A: 0 \leq \left\{m \phi\right\} \leq \frac{r}{n} \right\} = A \frac{r}{n} + \mathcal{O}(\log{r})$$
and
$$ Z= \# \left\{ -B \leq m \leq 0: 0 \leq \left\{m \phi\right\} \leq \frac{r}{n} \right\} = B \frac{r}{n} + \mathcal{O}(\log{r}).$$
Since $\left\{(-m) \phi \right\} = 1 - \left\{ m \phi \right\}$, the two estimates are completely equivalent and it suffices to prove the estimate for $Y$.
Let $F_k \leq A < F_{k+1}$. Then we know from the Lemma that the first $F_k$ elements are fairly evenly distributed and 
 $$ \# \left\{ 1 \leq m \leq F_k - 1: 0 \leq  \left\{m \phi\right\} \leq \frac{r}{n} \right\} = F_k \frac{r}{n} + \mathcal{O}(1).$$
It remains to understand the behavior of $F_k \leq m \leq A$ which can be written as
\begin{align*}
     Y &= \# \left\{ F_k \leq m \leq A: 0 \leq \left\{m \phi\right\} \leq \frac{r}{n} \right\} + F_k \frac{r}{n} + \mathcal{O}(1) \\
    &= \# \left\{ 0 \leq m \leq A - F_k: 0 \leq \left\{F_k \phi\right\}  + \left\{m \phi\right\} \leq \frac{r}{n} \right\} + F_k \frac{r}{n} + \mathcal{O}(1) \\
    &= \# \left\{ 0 \leq m \leq A - F_k: -  \left\{F_k \phi\right\} \leq   \left\{m \phi\right\} \leq \frac{r}{n} -  \left\{F_k \phi\right\}\right\} + F_k \frac{r}{n} + \mathcal{O}(1).
\end{align*}
Let us now use $F_{k_2}$ to denote the largest Fibonacci number such that
$F_{k_2} - 1 \leq A - F_k$
and use the Lemma once more to argue that
$$  \# \left\{ 0 \leq m \leq F_{k_2} -1: -  \left\{F_k \phi\right\} \leq   \left\{m \phi\right\} \leq \frac{r}{n} -  \left\{F_k \phi\right\}\right\} = F_{k_2} \frac{r}{n} + \mathcal{O}(1)$$
which then reduces the problem to analyzing
$$ \# \left\{ 0 \leq m \leq A - F_k - F_{k_2}: 0 \leq   \left\{m \phi\right\} + \left\{F_k \phi\right\} + \left\{F_{k_2} \phi\right\}\leq \frac{r}{n} \right\}.$$
It is clear that this procedure can be carried out $\sim \log{A} \leq \log{n}$ times and will result in an error of size $\mathcal{O}(\log{n})$. The goal is to refine this a little bit. Suppose we greedily remove the largest Fibonacci number $\ell$ consecutive times and end up with the range $0 \leq m \leq A- F_k - F_{k_2} - \dots - F_{\ell}$. Using the notation $k_1 = k$, 
this means that, after $\ell$ steps, the remaining set $S$ is
\begin{align*}
    S &= \left\{ 0 \leq m \leq A - \sum_{j=1}^{\ell} F_{k_{j}}: 0 \leq   \left\{m \phi\right\} + \sum_{j=1}^{\ell}  \left\{F_{k_j} \phi\right\}   \leq \frac{r}{n} \right\} \\
    &\subseteq  \left\{ 0 \leq m \leq F_{k_{\ell}}-1: 0 \leq   \left\{m \phi\right\} + \sum_{j=1}^{\ell}  \left\{F_{k_j} \phi\right\}   \leq \frac{r}{n} \right\}.
\end{align*}
At this point, we can invoke Lemma \ref{lem:twogap} and argue that once $\ell$ is so large that
$$ \frac{1}{F_{k_{\ell}}} \geq \frac{r}{n},$$
 we have $\# S = \mathcal{O}(1)$. Note that 
 $$  \frac{1}{F_{k_{\ell}}} \geq \frac{1}{F_{k- \ell}} \qquad \mbox{and thus the inequality is implied by} \qquad \frac{1}{F_{k- \ell}} \geq \frac{r}{n}.$$
 Recalling that $F_{k} \leq A < F_{k+1}$ and $F_{k-\ell} \sim \phi^{-\ell} F_k$ this shows that $\ell = \mathcal{O}(\log{r})$ is sufficient to deduce that
 $$ F_{k-\ell} \lesssim \frac{F_k}{r} \lesssim \frac{A}{r} \lesssim \frac{n}{r}.$$
 The desired estimate we then get is
 $$ \# Y = \left(\sum_{j=1}^{\ell} F_{k_{j}} \right) \frac{r}{n} + \mathcal{O}(\log{r}).$$
This is very close to the desired result except that the sum over the Fibonacci series has been truncated:  note that
$$ A - \sum_{j=1}^{\ell} F_{k_{j}} \leq 2 \cdot F_{k_{\ell}} \leq \frac{2A}{r} \leq \frac{2n}{r}$$
and thus
\begin{align*}
    \# Y &= \left(\sum_{j=1}^{\ell} F_{k_{j}} \right) \frac{r}{n} + \mathcal{O}(\log{r}) =  A \frac{r}{n} + \left(A-\sum_{j=1}^{\ell} F_{k_{j}} \right) \frac{r}{n} + \mathcal{O}(\log{r}) \\
    &= A \frac{r}{n}+ \mathcal{O}(\log{r})
\end{align*} 
as desired.
\end{proof}

\end{document}